\documentclass[12pt]{amsart}
\usepackage{amsmath,amssymb,amsfonts,amsthm,color}
\usepackage[all,cmtip]{xy}
\usepackage[top=1.5in, bottom=1.5in, left=1.2in, right=1.2in]{geometry}

\theoremstyle{plain}
\newtheorem{theorem}{Theorem}[section]
\newtheorem{corollary}[theorem]{Corollary}
\newtheorem{lemma}[theorem]{Lemma}
\newtheorem{proposition}[theorem]{Proposition}
\newtheorem{claim}[theorem]{Claim}
\theoremstyle{definition}
\newtheorem{definition}[theorem]{Definition}
\usepackage{url}

\DeclareMathOperator{\FS}{FS}

\newcommand{\red}[1]{\textcolor{red}{#1}}
\newcommand\myseq[1]{#1_1,\allowbreak #1_2,\allowbreak\dots}
\newcommand\bbN{\mathbb{N}}
\newcommand{\set}[2]{\{\,#1:#2\,\}}
\newcommand{\G}[1]{{\textcolor[rgb]{0.00,0.59,0.00}{#1}}}
\newcommand{\cF}{\mathcal{F}}

\title[Partition regularity of proper IP sets]{Partition regularity of proper infinite parallelepiped sets}
\author{Yonatan Gadot and Boaz Tsaban}
\address{Department of Mathematics, Bar-Ilan University, Ramat Gan, Israel}
\email{tsaban@math.biu.ac.il, yogadot710@gmail.com}

\begin{document}

\begin{abstract}
For a sequence $\myseq{a}$ in a semigroup, let $\FS(\myseq{a})$ be the set of all sums
  $a_{i_1}+\dotsb+a_{i_m}$, where $m\ge1$ and $i_1<\dotsb<i_m$.
  An \emph{infinite parallelepiped (IP)} set is a set that contains
  $\FS(\myseq{a})$ for some sequence $\myseq{a}$.
  Hindman's Finite Sums Theorem asserts that the IP sets of natural numbers
  are partition regular: for each finite coloring of an IP set of natural
  numbers, there is a monochromatic IP subset.
  Furstenberg generalized Hindman's theorem to arbitrary semigroups, but in this generality,
the guaranteed monochromatic IP set may be a singleton.

  A \emph{proper} infinite parallelepiped (IP) set in a semigroup is a set
  containing an \emph{infinite} set of the form $\FS(\myseq{a})$;
  equivalently, it is a set containing $\FS(\myseq{a})$ for some injective
  sequence $\myseq{a}$.
  We provide a complete classification of the semigroups for which the
  proper IP sets are partition regular, and show that this property is
  equivalent to other fundamental notions of additive Ramsey theory.
\end{abstract}



\maketitle

\section{Introduction}

Ramsey theory deals with the phenomenon that, whenever a rich mathematical structure is partitioned into finitely many parts,
at least one of the parts is mathematically rich.
A family of sets is \emph{partition regular} if, for each finite partition of a set from the family,
there is a part that is also in this family.
A \emph{finite coloring} of a set is a function from that set into a finite set of colors.
A subset of the colored set is \emph{monochromatic} if all of its members have the same color.
Thus, a family of sets that is closed under supersets is partition regular if and only if
for each finite coloring of a family member there is a monochromatic family member.

For clarity and compatibility of notation, we use additive notation for all semigroups,
\emph{including noncommutative ones}.
Thus, the results of this paper hold for both commutative and noncommutative semigroups.
Let $\bbN=\{1,2,3,\dotsc\}$ be the set of natural numbers.
For a sequence $\myseq{a}$ in a semigroup, let
\[
\FS(\myseq{a}) := \set{a_{i_{1}}+\dotsb+a_{i_m}}{m\ge 1, i_1< \dotsb<i_m}.
\]
(The notation $\FS$ stands for \emph{finite sums}.)
A subset of a semigroup is an \emph{infinite parallelepiped (IP)} set if it contains a set of the form $\FS(\myseq{a})$ for some
sequence $\myseq{a}$ in the semigroup.
For convenience, we do not require IP sets to be of the form
$\FS(\myseq{a})$; it suffices for them to contain a set of this form.
Hindman's celebrated Finite Sums Theorem~\cite[Theorem~3.1]{Hind} asserts that
for each finite coloring of the natural numbers, there is a monochromatic IP set.
Using the Stone-{\v{C}}ech compactification, Galvin and Glazer provided an elegant proof
of Hindman's theorem, which applies to all semigroups~\cite[Theorem~10.3]{GG}.
Furstenberg~\cite[Proposition~8.13]{Furst} attributes the following stronger assertion to Hindman,
who proved this assertion for $\bbN$, and credited this result to Galvin~\cite[Corollary~2.9]{Hindm}.

For a nonempty finite index set $F=\{i_1<\dotsb<i_m\}$, write
$a_F:=a_{i_1}+\dotsb+a_{i_m}$.
For nonempty finite index sets $F$ and $G$, write $F<G$ if $i<j$ for
every $i\in F$ and $j\in G$.
We use $n<F$ as an abbreviation for
$\{n\}<F$, and similarly for $>$, $\le$, and $\ge$.
A \emph{sumsequence} (traditionally called a \emph{sum subsystem})
of a sequence $\myseq{a}$ is a sequence $a_{F_1},a_{F_2},\dotsc$, where
$F_1<F_2<\dotsb$ are nonempty finite index sets.

\begin{theorem}
\label{thm:FGH}
For every sequence $\myseq{a}$ in a semigroup $S$ and every finite
coloring of $S$, there is a sumsequence $\myseq{b}$ of $\myseq{a}$
such that the set $\FS(\myseq{b})$ is monochromatic.
\end{theorem}

In particular, the family of IP sets in every semigroup is partition regular.

\medskip

In the semigroup $\bbN$ of natural numbers, every IP set contains a set $\FS(\myseq{a})$ for which the
sequence $\myseq{a}$ is \emph{injective}; that is, the terms of $\myseq{a}$ are distinct.
Thus, in $\bbN$, the term ``infinite parallelepiped'' is justified.
However, an IP set in a general semigroup may be a singleton!
Indeed, for each idempotent element $e$
in a semigroup, we have $\FS(e,e,\dotsc)=\{e\}$.
The following definition resolves this issue.

\begin{definition}
A \emph{proper IP} set in a semigroup is a set that contains a set $\FS(\myseq{a})$ for some
\emph{injective} sequence $\myseq{a}$ in the semigroup.
\end{definition}

For a sequence $\myseq{a}$ in a semigroup and a natural number $n$, let
\[
\FS(a_1,\dotsc,a_n):=
\set{a_F}{ \emptyset\neq F\subseteq\{1,\dotsc,n\} }.
\]

\begin{lemma}
For every sequence $\myseq{a}$, if the set $\FS(\myseq{a})$ is infinite, then the set
$\FS(\myseq{a})$ is a proper IP set, and so is every set containing it.
\end{lemma}
\begin{proof}
For every $m$, the decomposition
\begin{multline*}
\FS(\myseq{a}) = \FS(a_1,\dotsc,a_m)\cup
\FS(a_{m+1},a_{m+2},\dotsc)\\
\cup\bigl(\FS(a_1,\dotsc,a_m)+\FS(a_{m+1},a_{m+2},\dotsc)\bigr)
\end{multline*}
shows that the set $\FS(a_{m+1},a_{m+2},\dotsc)$ is infinite: otherwise all
three sets on the right would be finite.

We now choose successive nonempty finite index sets
$F_1<F_2<\dotsb$ such that the values $a_{F_1},a_{F_2},\dotsc$ are
distinct.
Having chosen $F_1<\dotsb<F_n$, choose a natural number $k>F_n$ if $n>0$,
and $k:=1$ if $n=0$.
The tail finite-sums set
$\FS(a_k,a_{k+1},\dotsc)$ is infinite, so it contains an element outside
the finite set $\{a_{F_1},\dotsc,a_{F_n}\}$.
Choose such an element and
a nonempty finite index set $F_{n+1}\ge k$ representing it.
Then the sequence $\myseq{b}$, defined by $b_n:=a_{F_n}$, is an
injective sumsequence of $\myseq{a}$, and associativity gives
$\FS(\myseq{b})\subseteq\FS(\myseq{a})$.
\end{proof}

In general, the proper IP sets need not be partition regular.
Consider the fan semilattice $(\bbN,\wedge)$, where
$m\wedge m=m$ and $m\wedge n=1$ for distinct $m,n$.
Color the central
vertex $1$ red and all other vertices green:
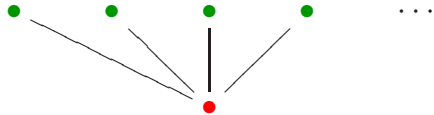
\[
\xymatrix{
\G{2} & \G{3} & \G{4} & \G{5} & \dotsm\\
& & \red{1}\ar@{-}[llu]\ar@{-}[lu]\ar@{-}[u]\ar@{-}[ur]
}
\]
The entire semilattice is a proper IP set, but it has no monochromatic proper IP subset, since every sum of distinct green elements is red.
It follows that, for the family of proper IP sets in a semigroup to be partition regular, the semigroup cannot contain the fan semilattice as a subsemigroup.

For which semigroups are the proper IP sets partition regular?
  This problem was raised, for example, by Andrews and Goldbring~\cite[Question~4.9]{AG}.
  We provide a complete solution.
  Our main result, Theorem~\ref{mr1}, states that the proper IP sets in a semigroup $S$ are partition regular if and only if $S$ has no subsemigroup of any of the following three types:
  (a) an infinite semigroup $T$ with $T+T$ finite;
  (b) the fan semilattice $(\bbN,\wedge)$, where $m\wedge n=1$ for distinct $m,n$; or
  (c) the semigroup $(\bbN\times\bbN)\cup\{0\}$ in which $0$ is absorbing and
  $(m_1,n_1)+(m_2,n_2)$ is $(m_1,n_1+n_2)$ if $m_1=m_2$, and $0$ otherwise.
  The theorem also answers several additional fundamental questions in the studied realm.

  The remainder of the paper is organized as follows.
  In Section~\ref{sec:proper-sumsequences}, we reduce an injective sequence with no proper sumsequence to a minimal one for which the intersection of the tail finite-sums sets is finite, and find a
  suitable right ideal in this finite semigroup.
  In Section~\ref{sec:partition-regularity}, these reductions yield a sumsequence with only finitely many sums involving at least two terms, from which we extract one of the three forbidden types and prove
  Theorem~\ref{mr1}.
  In Section~\ref{sec:moving}, we apply the classification to show that partition regularity of proper IP sets does not characterize moving semigroups.
  Finally, in Section~\ref{sec:manyIP}, we prove Theorem~\ref{thm:manyIP}: if every finite coloring of $S$ has a monochromatic proper IP set, then every such coloring has infinitely many pairwise
  disjoint monochromatic proper IP sets.

\section{Proper sumsequences}
\label{sec:proper-sumsequences}

This section reduces an injective sequence with no proper sumsequence to a
minimal one with finite tail intersection and an absorbing right ideal.

A sequence $\myseq{a}$ in a semigroup is
\emph{proper}~\cite[Definition~1.3]{T} if $a_{F_1}\neq a_{F_2}$ for all
nonempty finite index sets $F_1<F_2$.

An IP set containing a set $\FS(\myseq{a})$ generated by a proper sequence is necessarily a proper IP set, since every proper sequence is injective.
The converse fails in general.
Indeed, the fan semilattice is a proper IP set but contains no proper
sequence.
Thus, the two notions of properness are only partially related.
For groups, however, they coincide.

\begin{lemma}\label{lem:grp}
Every injective sequence in a group has a proper subsequence.
\end{lemma}

\begin{proof}
Let $\myseq{a}$ be an injective sequence in a group.
We construct a
proper subsequence $b_1:=a_{m_1},b_2:=a_{m_2},\dotsc$ by induction.
Set $m_1:=1$.
Once $m_1<\dotsb<m_n$ have been chosen, the set
$A_n:=\FS(b_1,\dotsc,b_n)\cup\{0\}$
is finite, and so is the set $-A_n+A_n$.
Thus, there is an element $b_{n+1}:=a_{m_{n+1}}$ in the
tail
$a_{m_n+1},a_{m_n+2},\dotsc$ that lies outside $-A_n+A_n$.

Let $F_1<F_2$ be nonempty finite index sets, and let $n+1$ be the
  largest member of $F_2$.
  Write $F_2=G\cup\{n+1\}$, where $G=\emptyset$ or $G<n+1$.
  Since $F_1\le n$ and $G\le n$, we have $b_{F_1}\in A_n$, and there
  is an element $u\in A_n$ such that $b_{F_2}=u+b_{n+1}$.
  If $b_{F_2}$ belonged to $A_n$, then
  the element
  \[
  b_{n+1}=-u+b_{F_2}=-u+(u+b_{n+1})
  \]
  would belong to $-A_n+A_n$, contrary to its choice.
  Therefore $b_{F_1}\ne b_{F_2}$, and the subsequence is proper.
\end{proof}

The sumsequence relation is transitive.
More explicitly, if
$b_i=a_{F_i}$ and $c_j=b_{G_j}$, then
$c_j=a_{\bigcup_{i\in G_j}F_i}$,
and the unions on the right are successive as $j$ increases.

\begin{definition}
The \emph{tail intersection} of a sequence $\myseq{a}$ is the set
\[
I(\myseq{a}):=\bigcap_{n=1}^{\infty}\FS(a_n,a_{n+1},\dotsc).
\]
\end{definition}

\begin{proposition}\label{prp:Tail}
A sequence $\myseq{a}$ in a semigroup has a proper sumsequence if and only if
it has a sumsequence $\myseq{b}$ with $I(\myseq{b})=\emptyset$.
\end{proposition}
\begin{proof}
($\Rightarrow$)
Let $\myseq{b}$ be a proper sumsequence of $\myseq{a}$.
If $x\in I(\myseq{b})$, choose a
nonempty finite index set $F$ with $x=b_F$.
Choose $n>F$.
Since $x$
belongs to the tail beginning at $n$, there is a nonempty finite index set
$G\ge n$, and hence $G>F$, with $x=b_G$, contradicting properness.
Thus
$I(\myseq{b})=\emptyset$.

$(\Leftarrow)$
We construct a proper subsequence $b_{m_1},b_{m_2},\dotsc$ of the sequence $\myseq{b}$.

Let $m_1:=1$.
Assume that we have defined $m_1,\dotsc,m_n$, and let
$P_n:=\FS(b_{m_1},\allowbreak\dotsc,\allowbreak b_{m_n})$.
For each element $x\in P_n$, the empty intersection of the tail finite-sums sets
provides an index $k(x)$ such
that $x\notin\FS(b_{k(x)},b_{k(x)+1},\dotsc)$.
Since the set $P_n$ is finite, we may choose a number $m_{n+1}>m_n$ such that
$m_{n+1}\ge k(x)$ for all $x\in P_n$.
Since the tail finite-sums sets are decreasing, we have
$P_n\cap\FS(b_{m_{n+1}},b_{m_{n+1}+1},\dotsc)=\emptyset$.

Let $c_1:=b_{m_1},c_2:=b_{m_2},\dotsc$.
For each $n$, we have
\[
\FS(c_1,\dotsc,c_n)\cap\FS(c_{n+1},c_{n+2},\dotsc)=\emptyset.
\]
Let $F_1<F_2$ be nonempty finite index sets, and choose $n$ with
$F_1\le n<F_2$.
Then $c_{F_1}\in \FS(c_1,\dotsc,c_n)$ and
$c_{F_2}\in \FS(c_{n+1},c_{n+2},\dotsc)$, and thus
$c_{F_1}\neq c_{F_2}$.
\end{proof}

\begin{lemma}\label{lem:subsemi}
Let $\myseq{a}$ be an injective sequence with no proper sumsequence,
and write $I:=I(\myseq{a})$.
Then:
\begin{enumerate}
\item The set $I$ is a subsemigroup.
\item Every sequence $\myseq{b}$ in $I$ is a sumsequence of
$\myseq{a}$.
\end{enumerate}
\end{lemma}
\begin{proof}
(1) By Proposition~\ref{prp:Tail}, we have $I\ne\emptyset$.
Let $s_1,s_2\in I$, and fix $n$.
Choose a nonempty finite set
$F_1\ge n$ with $s_1=a_{F_1}$, and then choose a nonempty finite set
$F_2>F_1$ with $s_2=a_{F_2}$.
We have
$s_1+s_2=a_{F_1\cup F_2}\in\FS(a_n,a_{n+1},\dotsc)$.
Since $n$ was arbitrary, $s_1+s_2\in I$.

(2) Let $\myseq{b}$ be a sequence in $I$.
Choose successive nonempty finite index sets recursively.
For $n=1$, choose a nonempty finite index set $F_1$ such that
$b_1=a_{F_1}$.
After $F_1<\dotsb<F_{n-1}$ have been chosen, choose $m>F_{n-1}$ and
represent $b_n$ in the tail beginning at $m$.
This gives a nonempty
finite set $F_n$ with $b_n=a_{F_n}$ and $F_{n-1}<F_n$.
Thus $\myseq{b}$ is a sumsequence of $\myseq{a}$.
\end{proof}

\begin{definition}
An injective sequence $\myseq{a}$ in a semigroup is \emph{minimal} if
\[
I(\myseq{b})=I(\myseq{a})
\]
for every injective sumsequence $\myseq{b}$ of $\myseq{a}$.
\end{definition}

\begin{lemma}\label{lem:index-set-extraction}
Let $\myseq{a}$ be a sequence in a semigroup $S$, $I:=I(\myseq{a})$, and $\cF$ be a family of nonempty finite
subsets of $\bbN$ such that:
\begin{enumerate}
\item for every natural number $m$, there is a set $F\in\cF$ with $F\ge m$;
\item $\{a_F:F\in\cF\}\cap I=\emptyset$.
\end{enumerate}
For every natural number $k$ and every finite set $K\subseteq S\setminus I$,
there are sets $F_1<F_2<\dotsb$ in $\cF$, with $F_1\ge k$, such
that $a_{F_1},a_{F_2},\dotsc$ is an injective sequence in
$S\setminus(I\cup K)$.
\end{lemma}

\begin{proof}
Suppose that $F_1<\dotsb<F_n$ have been chosen with distinct values
outside $K$, and let
$L_n:=K\cup\{a_{F_1},\dotsc,a_{F_n}\}$,
where the second set is empty when $n=0$.
By the hypotheses, $L_n\subseteq S\setminus I$.
For each $x\in L_n$, choose a natural number
$m(x)$ such that $x\notin\FS(a_{m(x)},a_{m(x)+1},\dotsc)$.
Choose
$m\ge k$ and $m\ge m(x)$ for every $x\in L_n$, with $m>F_n$ when
$n>0$.
Since the tail finite-sums sets decrease,
$L_n\cap\FS(a_m,a_{m+1},\dotsc)=\emptyset$.
Choose $F_{n+1}\in\cF$ with $F_{n+1}\ge m$.
Its value belongs
to the displayed tail finite-sums set, and hence avoids $K$ and all
previously selected values.
\end{proof}

\begin{lemma}\label{lem:monoid-proper}
Let $\myseq{a}$ be a minimal injective sequence.
If
$M:=\langle\myseq{a}\rangle$ is a monoid whose identity $e$ belongs
to $I(\myseq{a})$, then $\myseq{a}$ has a proper subsequence.
\end{lemma}

\begin{proof}
Since $\myseq{a}$ is injective, at most one of its terms is equal to
$e$.
By passing to a tail, we may therefore assume that $a_n\ne e$ for
every $n$.
Minimality and transitivity of the sumsequence relation show that this
tail is minimal and has the same tail intersection as $\myseq{a}$.
Since $e$ belongs to this tail intersection, it belongs to the
subsemigroup generated by the tail and is its identity.
Relabel the tail as $\myseq{a}$ and its generated monoid as $M$.

Since $e\in I(\myseq{a})$, choose successive nonempty finite sets
$E_1<E_2<\dotsb$ such that $a_{E_j}=e$ for every $j$.
Each $E_j$ has at least two elements, since $a_n\ne e$ for every
$n$.
Write $E_j=\{i_j\}\cup E'_j$ and $i_j<E'_j$,
where $E'_j$ is nonempty.
Then $a_{i_j}+a_{E'_j}=e$,
so $s_j:=a_{i_j}$ has a right inverse in $M$.

The sequence $\myseq{s}$ is an injective sumsequence of the minimal
sequence $\myseq{a}$.
Consequently, $I(\myseq{s})=I(\myseq{a})$, and in particular
$e\in I(\myseq{s})$.
Choose successively nonempty finite sets $K_1<K_2<\dotsb$ such that
$s_{K_j}=e$ for every $j$.
Each $K_j$ has at least two elements, since no term of $\myseq{s}$ is
equal to $e$.
Write
$K_j=K'_j\cup\{m_j\}$ and $K'_j<m_j$,
where $K'_j$ is nonempty.
Then
$s_{K'_j}+s_{m_j}=e$,
so $s_{m_j}$ has a left inverse.
It also has a right inverse, and a left inverse and a right inverse of
the same element in a monoid coincide.
Thus $s_{m_1},s_{m_2},\dotsc$ is an injective sequence in the group
of units of $M$.

By Lemma~\ref{lem:grp}, this sequence has a proper subsequence.
\end{proof}

\begin{lemma}\label{lem:finite-tail}
Let $\myseq{a}$ be an injective sequence with no proper sumsequence.
Then $\myseq{a}$ has an injective
sumsequence $\myseq{b}$ for which the set $I(\myseq{b})$ is finite.
\end{lemma}

\begin{proof}
By Lemma~\ref{lem:subsemi}, the set $I(\myseq{a})$ is a subsemigroup.
If it is finite, take $\myseq{b}:=\myseq{a}$.
Suppose that it
is infinite.
By a theorem of Shevrin~\cite{Shev}, it contains a
subsemigroup of one of the following types:
\begin{enumerate}
\item $(\bbN,+)$;
\item an infinite periodic group (all cyclic subgroups are finite);
\item an infinite right or left zero semigroup ($x+y=y$ or $x+y=x$,
  respectively);
\item the minimum semilattice $(\bbN,\wedge)$;
\item the maximum semilattice $(\bbN,\vee)$;
\item an infinite semigroup $T$ with $T+T$ finite;
\item the fan semilattice $(\bbN,\wedge)$, where $m\wedge n=1$ for
  distinct $m,n$.
\end{enumerate}
The first five types have proper sequences: use $2,4,8,\dotsc$ in
$(\bbN,+)$; use Lemma~\ref{lem:grp} in a periodic group; use any
injective sequence in a zero semigroup; and use $1,2,3,\dotsc$ in the
two linearly ordered semilattices.
Lemma~\ref{lem:subsemi}(2) would turn
such a sequence into a proper sumsequence of $\myseq{a}$, a
contradiction.

In type~(6), choose an injective sequence $\myseq{b}$ in the infinite
set $T\setminus(T+T)$.
If $x\in I(\myseq{b})$, move to a tail that does
not contain $x$ as a term.
Every representation of $x$ in that tail
then uses at least two terms, so $x\in T+T$.
Hence
$I(\myseq{b})\subseteq T+T$ is finite.
In type~(7), take
$b_n:=n+1$; then $I(\myseq{b})=\{1\}$.
In both cases the terms of
$\myseq{b}$ lie in $I(\myseq{a})$, so Lemma~\ref{lem:subsemi}(2) makes
$\myseq{b}$ a sumsequence of $\myseq{a}$.
\end{proof}

\begin{proposition}\label{prp:minsubsemi}
Let $S$ be a semigroup, and let $\myseq{a}$ be an injective sequence
with no proper sumsequence.
There is an injective sumsequence
$\myseq{b}$ of $\myseq{a}$ such that $I(\myseq{b})$ is a finite
subsemigroup of $S$ and the sequence $\myseq{b}$ is minimal.
\end{proposition}

\begin{proof}
By Lemma~\ref{lem:finite-tail}, choose an injective sumsequence
$\myseq{c}$ of $\myseq{a}$ with $I(\myseq{c})$ finite.
For every
injective sumsequence $\myseq{d}$ of $\myseq{c}$, we have
$I(\myseq{d})\subseteq I(\myseq{c})$.
The sequence $\myseq{d}$ has no
proper sumsequence, and Lemma~\ref{lem:subsemi} shows that
$I(\myseq{d})$ is a nonempty subsemigroup.
Choose $\myseq{b}$ so that
$I(\myseq{b})$ is minimal under inclusion among these finitely many
possibilities.

If $\myseq{e}$ is an injective sumsequence of $\myseq{b}$, then it is
an injective sumsequence of $\myseq{c}$ and
$I(\myseq{e})\subseteq I(\myseq{b})$.
The choice of $\myseq{b}$ forces
equality, so $\myseq{b}$ is minimal.
Transitivity makes it a
sumsequence of $\myseq{a}$, and Lemma~\ref{lem:subsemi} gives the
subsemigroup assertion.
\end{proof}

\begin{proposition}
\label{prp:Rideal}
Assume that an injective sequence $\myseq{a}$ in a semigroup is minimal
and that the set $I:=I(\myseq{a})$ is a finite subsemigroup.
Then there exist
a right ideal $R$ of $I$ and an injective sumsequence $\myseq{b}$ of
$\myseq{a}$ such that
\[
R+\langle\myseq{b}\rangle\subseteq R.
\]
\end{proposition}

\begin{proof}
First suppose that there are an idempotent $e\in I$ and an injective
sumsequence $\myseq{b}$ of $\myseq{a}$ such that
$e+\FS(\myseq{b})\subseteq I$.
Let $R:=e+I$.
Then $R+I\subseteq R$, so $R$ is a right ideal of $I$.
Minimality gives $I(\myseq{b})=I$.
Choose $m$ such that
$I\subseteq\FS(b_1,\dotsc,b_m)$.
For $k>m$,
\[
R+b_k\subseteq e+\FS(b_1,\dotsc,b_m)+b_k
\subseteq e+\FS(\myseq{b})\subseteq I.
\]
Since $e+R=R$, it follows that $R+b_k\subseteq R$.
Repeated application shows that
$R+\langle b_{m+1},b_{m+2},\dotsc\rangle\subseteq R$, which proves
the assertion in this case.

Suppose now that no such $e$ and $\myseq{b}$ exist.

\begin{claim}\label{clm:all-left-identities}
There is an injective sumsequence $\myseq{c}$ of $\myseq{a}$ such that
$c_n\notin I$ for every $n$, and every idempotent of $I$ is a left
identity for every $c_n$.
\end{claim}

\begin{proof}[Proof of the claim]
The nonempty finite semigroup $I$ has an idempotent.
Pass
successively through its finitely many idempotents.
Suppose that
$\myseq{z}$ is the current injective sumsequence and that $e$ is the
next idempotent.
Minimality gives $I(\myseq{z})=I$.
Let
$\cF(e)$ consist of all sets $G\cup H$, where $G<H$ are
nonempty finite index sets satisfying
$z_G=e$ and $e+z_H\notin I$.
For every $m\in\bbN$, there is an
$F\in\cF(e)$ with $F\ge m$.
Indeed, since $e\in I(\myseq{z})=I$, choose a nonempty finite index
set $G\ge m$ such that $z_G=e$.
Apply the present case assumption to the tail following $G$ to choose
a nonempty finite index set $H>G$ such that $e+z_H\notin I$.
Then $F:=G\cup H$ belongs to $\cF(e)$ and satisfies
$F\ge m$.
For $F=G\cup H$ in this family, we have $z_F=e+z_H\notin I$ and $e+z_F=z_F$.
Lemma~\ref{lem:index-set-extraction} supplies an injective sumsequence, all
of whose terms lie outside $I$ and have $e$ as a left identity.

The construction can be continued because every new sequence is an
injective sumsequence of the original minimal sequence.
A left-identity
property already obtained survives passage to a sumsequence, since it
holds for every finite sum of current terms.
After finitely many stages
we obtain the required sequence $\myseq{c}$.
\end{proof}

Fix a sequence $\myseq{c}$ as in Claim~\ref{clm:all-left-identities}.
Since it is
an injective sumsequence of the minimal sequence $\myseq{a}$, it is
minimal and $I(\myseq{c})=I$.

\begin{claim}\label{clm:eventual-right-absorption}
For every idempotent $e\in I$, there is a natural number $m$ such that
\[
\FS(c_m,c_{m+1},\dotsc)+e\subseteq I.
\]
\end{claim}

\begin{proof}[Proof of the claim]
Suppose otherwise for some idempotent $e$.
Let
$\cF(e)$ consist of all sets $H\cup G$, where
$H<G$ are finite and nonempty,
$c_H+e\notin I$, and $c_G=e$.
For every $m\in\bbN$, there is an $F\in\cF(e)$ with $F\ge m$.
Indeed, choose a nonempty finite set $H\ge m$ such that
$c_H+e\notin I$, and then choose a nonempty finite set $G>H$ such
that $c_G=e$.
Then $H\cup G\in\cF(e)$ and $H\cup G\ge m$.
For every $F=H\cup G$ in the family,
$c_F=c_H+e\notin I$ and $c_F+e=c_F$.
Lemma~\ref{lem:index-set-extraction} gives an injective sumsequence
$\myseq{d}$ with these properties.
The element $e$ is also a left
identity for every $d_n$.
By minimality and transitivity, the sequence $\myseq{d}$ is minimal and
$I(\myseq{d})=I$.
Moreover,
$e\in I(\myseq{d})\subseteq\langle\myseq{d}\rangle$.
Since $e$ is both a left and a right identity for every generator,
$\langle\myseq{d}\rangle$ is a monoid with identity $e$.
Lemma~\ref{lem:monoid-proper} gives a proper subsequence
$\myseq{t}$ of $\myseq{d}$.
Minimality gives $I(\myseq{t})=I$,
whereas Proposition~\ref{prp:Tail} gives $I(\myseq{t})=\emptyset$,
a contradiction.
\end{proof}

\begin{claim}\label{clm:tail-absorption-properties}
There is a natural number $m$ such that
\begin{equation}
\label{tail-plus-I}
\FS(c_m,c_{m+1},\dotsc)+I\subseteq I,
\end{equation}
and, for every $n$ and every $s\in I$,
\begin{equation}
\label{left-test}
s+c_n\in I\text{ implies } c_n\in I.
\end{equation}
\end{claim}

\begin{proof}[Proof of the claim]
There are finitely many idempotents, so Claim~\ref{clm:eventual-right-absorption}
gives one index $m$ that works for all of them.
Fix an idempotent
$e\in I$.
It is a left identity for every finite sum of terms of
$\myseq{c}$, and hence for every element of $I$.
For $x\in\FS(c_m,c_{m+1},\dotsc)$ and $y\in I$, we have
\[
x+y=x+(e+y)=(x+e)+y\in I+I\subseteq I,
\]
which proves Equation~\eqref{tail-plus-I}.

Now suppose that $s+c_n\in I$.
The finite cyclic subsemigroup generated
by $s$ contains an idempotent $k\cdot s$.
If $k=1$, then $s$ is a left
identity for $c_n$, so $c_n=s+c_n\in I$.
If $k>1$, then
$c_n=k\cdot s+c_n=(k-1)\cdot s+(s+c_n)\in I$.
This proves Equation~\eqref{left-test}.
\end{proof}

\begin{claim}\label{clm:tail-absorption-proper}
The sequence $\myseq{c}$ has a proper sumsequence.
\end{claim}

\begin{proof}[Proof of the claim]
Choose $k\ge m$ such that $c_n\notin I$ for every $n\ge k$, where $m$
is as in Claim~\ref{clm:tail-absorption-properties}, and fix $u\in I$.
Let $\cF(u)$ consist of all sets $E\cup\{n\}$ with $E$
nonempty and finite,
$E<n$, $c_E=u$, and $n\ge k$.
For every $i\in\bbN$, there is an $F\in\cF(u)$ with $F\ge i$.
Indeed, given $i$, choose a nonempty
finite set $\{i,k\}\le E$ with $c_E=u$, and then choose $n>E$.
For
every member $F=E\cup\{n\}$, the contrapositive of
Equation~\eqref{left-test} gives $c_F=u+c_n\notin I$.
Lemma~\ref{lem:index-set-extraction}, with initial index $k$, supplies
successive members $F_j=E_j\cup\{n_j\}$ with $F_1\ge k$ such that
$d_j:=c_{F_j}$ is injective.

Let $H=\{i_1<\dotsb<i_j\}$ be nonempty, and define $t_H:=u$ if
$j=1$, and $t_H:=d_{i_1}+\dotsb+d_{i_{j-1}}+u$ if $j>1$.
In the latter case the initial sum belongs to
$\FS(c_k,c_{k+1},\dotsc)$, so Equation~\eqref{tail-plus-I} gives
$t_H\in I$; this is also true when $j=1$.
Since
$d_H=t_H+c_{n_{i_j}}$
and $c_{n_{i_j}}\notin I$, Equation~\eqref{left-test} gives
$d_H\notin I$.
Thus $\FS(\myseq{d})\cap I=\emptyset$.

For every $m\in\bbN$, all but finitely many $F_j$ satisfy
$F_j\ge m$, so $I(\myseq{d})\subseteq I$.
Hence
$I(\myseq{d})=\emptyset$, and Proposition~\ref{prp:Tail} gives a
proper sumsequence of $\myseq{d}$, and therefore of $\myseq{c}$.
\end{proof}

Claim~\ref{clm:tail-absorption-proper} gives a proper sumsequence
$\myseq{t}$ of $\myseq{c}$.
Since $\myseq{t}$ is an injective
sumsequence of the minimal sequence $\myseq{a}$, minimality gives
$I(\myseq{t})=I$, whereas Proposition~\ref{prp:Tail} gives
$I(\myseq{t})=\emptyset$.
This contradiction completes the proof of Proposition~\ref{prp:Rideal}.
\end{proof}

\section{Partition regularity of proper IP sets and related notions}
\label{sec:partition-regularity}

Using the reductions of Section~\ref{sec:proper-sumsequences}, we derive the
three forbidden subsemigroups and prove the main classification theorem.

For a sequence $\myseq{a}$ in a semigroup, let
\[
\FS_{\ge2}(\myseq{a}):=\set{a_F}{F\subseteq\bbN\text{ finite},\ |F|\ge2}.
\]

\begin{definition}\label{def:three-types}
Let $\myseq{c}$ be an injective sequence in a semigroup $S$, and let
$R\subseteq I\subseteq S$.
We use the following terminology:
\begin{enumerate}
\item The sequence $\myseq{c}$ is of \emph{type~1} relative to $(I,R)$
if there are $x\in\FS(\myseq{c})\setminus I$ and $n$ such that
\[
(x+\FS(c_n,c_{n+1},\dotsc))\cap R=\emptyset;
\]
\item The sequence $\myseq{c}$ is of \emph{type~2} relative to $(I,R)$
if there are $x\in\FS(\myseq{c})\setminus I$ and $n$ such that
\[
\emptyset\ne
\set{y\in\FS(c_n,c_{n+1},\dotsc)}{x+y\in R}
\subseteq I;
\]
\item The sequence $\myseq{c}$ is of \emph{type~3} relative to $(I,R)$
if, for every $x\in\FS(\myseq{c})\setminus I$ and every $n$,
there is
$y\in\FS(c_n,c_{n+1},\dotsc)\setminus I$
such that $x+y\in R$.
\end{enumerate}
\end{definition}

The three alternatives are exhaustive.
Indeed, if type~1 fails, the
set occurring in type~2 is nonempty for every admissible $x$ and $n$;
if type~2 also fails, that set always contains an element outside $I$,
which is type~3.

\begin{lemma}\label{lem:maximal-reduction}
Let $\myseq{a}$ be an injective sequence with no
proper sumsequence.
There is an injective sumsequence $\myseq{b}$ of
$\myseq{a}$ such that the set $I:=I(\myseq{b})$ is a finite subsemigroup with a
nonempty right ideal $R$ satisfying:
\begin{enumerate}
\item $\myseq{b}$ is minimal;
\item $R+\langle\myseq{b}\rangle\subseteq R$;
\item if $J$ is a right ideal of $I$ with $R\subseteq J$, and some
injective sumsequence $\myseq{c}$ of $\myseq{b}$ satisfies
$J+\langle\myseq{c}\rangle\subseteq J$, then $J=R$.
\end{enumerate}
\end{lemma}

\begin{proof}
Proposition~\ref{prp:minsubsemi} gives a minimal injective sumsequence
$\myseq{c}$ of $\myseq{a}$ such that $I(\myseq{c})$ is finite.
Proposition~\ref{prp:Rideal} gives a right ideal $R$ of this finite
semigroup and an injective sumsequence $\myseq{s}$ of $\myseq{c}$ such
that
$R+\langle\myseq{s}\rangle\subseteq R$.
The sequence $\myseq{s}$ is minimal and has the same tail intersection,
by minimality and transitivity.

Consider the right ideals $J$ of $I(\myseq{s})$ that contain $R$ and
for which there is an injective sumsequence $\myseq{d}$ of $\myseq{s}$
with
$J+\langle\myseq{d}\rangle\subseteq J$.
This nonempty family is finite.
Choose a maximal member $J$ and a
witnessing sequence $\myseq{b}$.
Minimality and transitivity give~(1),
and the witnessing property gives~(2), with $J$ in place of $R$.
If a larger right ideal were witnessed by an injective sumsequence of
$\myseq{b}$, transitivity of sumsequences would place it in the family
just considered, contradicting the maximality of $J$.
Thus~(3) holds.
Relabeling $J$ as $R$ completes the proof.
\end{proof}

\begin{lemma}\label{lem:fs2}
Let $\myseq{a}$ be an injective sequence with no
proper sumsequence.
There is an injective sumsequence $\myseq{b}$ of
$\myseq{a}$ with $\FS_{\ge2}(\myseq{b})$ finite.
\end{lemma}

\begin{proof}
Lemma~\ref{lem:maximal-reduction} gives an injective sumsequence of the
original sequence, a finite tail intersection $I$, and a nonempty right
ideal $R$ satisfying minimality, absorption, and the stated maximality
property.
By transitivity, it suffices to find the required sumsequence
of this new sequence.
Replace the original sequence by it and relabel it
as $\myseq{a}$.

\begin{claim}\label{clm:hereditary-types}
There is an injective sumsequence $\myseq{c}$ of $\myseq{a}$ such that:
\begin{enumerate}
\item $\myseq{c}$ is minimal and $I(\myseq{c})=I$;
\item $R+\langle\myseq{c}\rangle\subseteq R$;
\item if $J$ is a right ideal of $I$ with $R\subseteq J$, and an
injective sumsequence $\myseq{d}$ of $\myseq{c}$ satisfies
$J+\langle\myseq{d}\rangle\subseteq J$, then $J=R$;
\item one of the following holds:
\begin{enumerate}
\item every injective sumsequence of $\myseq{c}$ has a further
injective sumsequence of type~1 relative to $(I,R)$;
\item every injective sumsequence of $\myseq{c}$ has a further
injective sumsequence of type~2 relative to $(I,R)$;
\item every injective sumsequence of $\myseq{c}$ is of type~3 relative
to $(I,R)$.
\end{enumerate}
\end{enumerate}
\end{claim}

\begin{proof}[Proof of the claim]
A sequence is regarded as a sumsequence of itself, using index sets that are singletons.
If every injective sumsequence of $\myseq{a}$ has a further one
of type~1, take $\myseq{c}:=\myseq{a}$.
Otherwise, pass to an injective
sumsequence $\myseq{z}$ having no injective sumsequence of type~1.
If
every injective sumsequence of $\myseq{z}$ has a further one of type~2,
take $\myseq{c}:=\myseq{z}$.
Otherwise, pass to an injective
sumsequence $\myseq{c}$ of $\myseq{z}$ having no injective sumsequence
of type~2.
It also has no injective sumsequence of type~1.
By the
exhaustiveness of Definition~\ref{def:three-types}, every injective
sumsequence of $\myseq{c}$ is therefore of type~3.

Minimality and transitivity preserve minimality and the tail
intersection.
The absorption property is preserved because the
generated subsemigroup can only decrease.
Finally, a larger right ideal
witnessed after the passage would also have been witnessed before it,
by transitivity of sumsequences.
Thus all standing assumptions remain
true.
\end{proof}

Fix a sequence $\myseq{c}$ as in Claim~\ref{clm:hereditary-types}, replace
$\myseq{a}$ by $\myseq{c}$, and relabel.
Thus the standing properties
and one of the three alternatives in that claim hold for the sequence $\myseq{a}$.

\begin{claim}\label{clm:type1-consequence}
One of the second and third alternatives holds.
\end{claim}

\begin{proof}[Proof of the claim]
Suppose otherwise.
We construct an injective sumsequence $\myseq{c}$ of $\myseq{a}$,
together with nested injective sumsequences
$\myseq{b^{(1)}},\myseq{b^{(2)}},\dotsc$.
Set $\myseq{b^{(1)}}:=\myseq{a}$.
Suppose that $c_1,\dotsc,c_{n-1}$ and $\myseq{b^{(n)}}$ have been
chosen.
Minimality gives $I(\myseq{b^{(n)}})=I$.
Move to a tail whose
finite-sums set avoids $c_1,\dotsc,c_{n-1}$, and choose an injective
type~1 sumsequence $\myseq{s}$ of that tail.
There are a nonempty
finite set $F$ and an index $j$ such that
$s_F\notin I$ and $(s_F+\FS(s_j,s_{j+1},\dotsc))\cap R=\emptyset$.
Increasing $j$ if necessary, arrange $F<j$, and let
$c_n:=s_F$ and $\myseq{b^{(n+1)}}:=s_j,s_{j+1},\dotsc$.
The tail choice makes the $c_n$ distinct, and the inequalities $F<j$
imposed at each stage, together with nesting, make $\myseq{c}$ a
sumsequence of $\myseq{a}$.

A singleton sum from $\myseq{c}$ lies outside $I\supseteq R$.
If
$i_1<\dotsb<i_k$ and $k\ge2$, the construction gives
$c_{i_2}+\dotsb+c_{i_k}\in\FS(\myseq{b^{(i_1+1)}})$,
so the defining choice at stage $i_1$ gives
$c_{i_1}+\dotsb+c_{i_k}\notin R$.
Thus
$\FS(\myseq{c})\cap R=\emptyset$.
On the other hand, minimality gives
$\emptyset\ne R\subseteq I(\myseq{c})\subseteq\FS(\myseq{c})$, a
contradiction.
\end{proof}

\begin{claim}\label{clm:type2-diagonal}
If the second alternative holds, then the sequence $\myseq{a}$ has an injective
sumsequence $\myseq{b}$ such that for all $n$ we have $b_n\notin I$ and
\begin{equation}
\label{type2property}
\emptyset\ne
\set{y\in\FS(b_{n+1},b_{n+2},\dotsc)}{b_n+y\in R}
\subseteq I.
\end{equation}
\end{claim}

\begin{proof}[Proof of the claim]
We construct an injective sumsequence $\myseq{b}$ of $\myseq{a}$,
together with nested injective sumsequences
$\myseq{c^{(1)}},\myseq{c^{(2)}},\dotsc$.
Set $\myseq{c^{(1)}}:=\myseq{a}$.
Suppose that $b_1,\dotsc,b_{n-1}$ and $\myseq{c^{(n)}}$ have been
chosen.
Since
$I(\myseq{c^{(n)}})=I$, move to a tail whose finite-sums set avoids the
previous $b_i$, and choose an injective type~2 sumsequence $\myseq{s}$
of that tail.
Thus there are a nonempty finite set $F$ and an index $m$ such that
$s_F\notin I$ and the set
\[
Y:=\set{y\in\FS(s_m,s_{m+1},\dotsc)}{s_F+y\in R}
\]
is nonempty and contained in $I$.
Choose $u_n\in Y$.
Since $I(\myseq{s})=I$, choose
$F\cup\{m\}<j$ with $u_n\in\FS(s_j,s_{j+1},\dotsc)$, and let
$b_n:=s_F$ and $\myseq{c^{(n+1)}}:=s_j,s_{j+1},\dotsc$.
The inequalities $F\cup\{m\}<j$ imposed at successive stages make
$\myseq{b}$ a sumsequence, and tail avoidance makes it injective.

For each $n$, every member of
$\FS(b_{n+1},b_{n+2},\dotsc)$ belongs to
$\FS(\myseq{c^{(n+1)}})$, so the set in
Equation~\eqref{type2property} is contained in $Y$.
The tail
$b_{n+1},b_{n+2},\dotsc$ is an injective sumsequence of the original
minimal sequence, and hence has tail intersection $I$.
Since $u_n\in I$, we have
$u_n\in\FS(b_{n+1},b_{n+2},\dotsc)$.
Since $b_n+u_n\in R$, the set is nonempty.
\end{proof}

\begin{claim}\label{clm:type2-enlargement}
Let $\myseq{b}$ be an injective sumsequence of $\myseq{a}$ whose terms
lie outside $I$ and for which Equation~\eqref{type2property} holds.
Then there are an index $m$ and a right ideal $J$ of $I$ such that
$R\subsetneq J$ and
\[
J+\langle b_m,b_{m+1},\dotsc\rangle\subseteq J.
\]
\end{claim}

\begin{proof}[Proof of the claim]
Equation~\eqref{type2property} for $n=1$ gives a nonempty finite set
$G\subseteq\{2,3,\dotsc\}$ such that $b_1+b_G\in R$.
Write
$\{1\}\cup G=\{i_1<\dotsb<i_k\}$.
The entire sum belongs to $R$,
whereas $b_{i_k}\notin I\supseteq R$.
Choose the largest $j<k$ such
that
$b_{i_j}+\dotsb+b_{i_k}\in R$.
Then
$v:=b_{i_{j+1}}+\dotsb+b_{i_k}\notin R$,
and Equation~\eqref{type2property}, applied with $n=i_j$, gives
$v\in I$.
Let $n:=i_j$ and
$R':=\set{s\in I}{b_n+s\in R}$.
This is a right ideal of $I$ and $v\in R'\setminus R$.

Represent every member of the finite set $R'$ in
$\FS(b_{n+1},b_{n+2},\dotsc)$, and choose $m>n$ beyond all indices
used.
If $s\in R'$ and $x\in\FS(b_m,b_{m+1},\dotsc)$, then
$b_n+(s+x)=(b_n+s)+x\in R$.
Equation~\eqref{type2property} gives
$s+x\in I$, and hence $s+x\in R'$.
In particular,
$R'+b_i\subseteq R'$ for every $i\ge m$.
Induction on word length
shows that $R'$ absorbs the semigroup generated by this tail.
Set
$J:=R\cup R'$.
The union of right ideals is a right ideal, both parts
absorb the tail, and $v\in J\setminus R$.
\end{proof}

\begin{claim}\label{clm:type3-finite-sums}
If the third alternative holds, then the sequence $\myseq{a}$ has an injective
sumsequence $\myseq{b}$ such that
$\FS_{\ge2}(\myseq{b})\subseteq R$.
\end{claim}

\begin{proof}[Proof of the claim]
Pass to a tail with no term in $I$, and call it
$\myseq{c^{(1)}}$.
At stage $n$, let $b_n$ be the first term of
$\myseq{c^{(n)}}$, and let $\myseq{z}$ be the remaining tail.
The sequence $\myseq{c^{(n)}}$ is of type~3, so, for every
$m\in\bbN$, the family of nonempty finite sets $F$ satisfying
$z_F\notin I$ and $b_n+z_F\in R$ contains a set $F\ge m$.
Lemma~\ref{lem:index-set-extraction}, applied while avoiding
$b_1,\dotsc,b_n$, gives an injective sumsequence
$\myseq{c^{(n+1)}}$ of $\myseq{z}$, all of whose terms lie outside
$I$, and
\[
b_n+\langle\myseq{c^{(n+1)}}\rangle\subseteq R.
\]
Indeed, the first added term places the sum in $R$, and all remaining
terms lie in $\langle\myseq{a}\rangle$, which $R$ absorbs.

The nested construction makes $\myseq{b}$ a sumsequence, and the
avoidance condition makes it injective.
If $i_1<\dotsb<i_k$ and $k\ge2$, then
$b_{i_2}+\dotsb+b_{i_k}$ belongs to
$\langle\myseq{c^{(i_1+1)}}\rangle$, and hence
$b_{i_1}+\dotsb+b_{i_k}\in R$.
\end{proof}

By Claim~\ref{clm:type1-consequence}, only the second and third
alternatives remain.
If the second alternative holds, Claim~\ref{clm:type2-diagonal} gives an
injective sumsequence $\myseq{b}$ satisfying
Equation~\eqref{type2property}, with all terms outside $I$.
Minimality
and transitivity give $I(\myseq{b})=I$, and the absorption property is
inherited.
Claim~\ref{clm:type2-enlargement} gives a right ideal
$J\supsetneq R$ that absorbs a tail of $\myseq{b}$, contradicting the
maximality of $R$.

The third alternative therefore holds.
Claim~\ref{clm:type3-finite-sums}
gives an injective sumsequence $\myseq{b}$ with
$\FS_{\ge2}(\myseq{b})\subseteq R$.
Since $R$ is finite and all
passages were to sumsequences of the original sequence, this proves the
assertion.
\end{proof}

\begin{lemma}\label{lem:stable-ideal-tail}
Let $\myseq{c}$ be an injective sequence such that
$\FS_{\ge2}(\myseq{c})$ is finite, and let $I:=I(\myseq{c})$.
Then the sequence
$\myseq{c}$ has a tail $\myseq{d}$ such that
\[
\FS_{\ge2}(d_n,d_{n+1},\dotsc)=I
\]
for all $n$, and $I$ is a two-sided ideal of the semigroup $\langle\myseq{d}\rangle$.
\end{lemma}

\begin{proof}
Let $D_n:=\FS_{\ge2}(c_n,c_{n+1},\dotsc)$.
These sets decrease inside the finite set $D_1$.
We claim that their intersection is $I$.
The inclusion $\bigcap_{n=1}^{\infty}D_n\subseteq I$ is immediate.
Conversely, fix $s\in I$ and $n\in\bbN$.
Move beyond both $n$ and the unique possible occurrence of $s$ as a
term, and represent $s$ in that later tail.
This representation uses at least two terms, so $s\in D_n$.
Thus $(D_n)$ stabilizes at $I$.
Pass to a tail for which all the displayed equalities hold, and relabel
it as $\myseq{c}$.

For each $s\in I$, choose a finite-sum representation of $s$.
Since
$I$ is finite, there is an $m$ such that all these representations use
only $c_1,\dotsc,c_{m-1}$.
Concatenating one of them with a finite sum from the tail gives
$I+\FS(c_m,c_{m+1},\dotsc)\subseteq I$.
For $k\ge m$ and $s\in I$, represent $s$ in a tail beginning after
$c_k$.
Then
$c_k+s\in\FS_{\ge2}(c_k,c_{k+1},\dotsc)=I$.
Also $I\subseteq\FS(c_m,c_{m+1},\dotsc)$.
Repeated application of the
two inclusions absorbs every word on either side.
The tail beginning at
$m$ has all the asserted properties.
\end{proof}

\begin{lemma}\label{lem:length-determined}
Let $\myseq{c}$ be a minimal injective sequence.
Assume that the set $I:=I(\myseq{c})$ is a finite two-sided ideal of
$\langle\myseq{c}\rangle$ and
\[
\FS_{\ge2}(c_n,c_{n+1},\dotsc)=I
\]
for all $n$.
Then the sequence $\myseq{c}$ has an injective subsequence $\myseq{b}$ such that:
\begin{enumerate}
\item every increasing ordered sum of at least two terms of $\myseq{b}$
  is determined solely by its length;
\item $\FS_{\ge2}(b_n,b_{n+1},\dotsc)=I$ for every $n$;
\item $I$ is a two-sided ideal of $\langle\myseq{b}\rangle$.
\end{enumerate}
\end{lemma}

\begin{proof}
Color pairs $\{i,j\}$, with $i<j$, by $c_i+c_j\in I$.
Ramsey's theorem
gives a subsequence $c_{n_1},c_{n_2},\dotsc$ and $s_2\in I$ such that
$c_{n_i}+c_{n_j}=s_2$ for $i<j$.
Since $I$ is an ideal, all
$c_{n_i}+s_2$ lie in the finite set $I$.
Pass to a further subsequence
$\myseq{b}$ and choose $s_3\in I$ such that
$b_i+s_2=s_3$ for every $i$.

An increasing sum of $2k$ terms is $k\cdot s_2$.
An increasing sum of
$2k+1$ terms, with $k\ge1$, is
$s_3+(k-1)\cdot s_2$.
This proves (1).
For (2), the forward inclusion follows from the
corresponding property of $\myseq{c}$.
For the reverse inclusion, fix
$s\in I$ and an arbitrary tail of $\myseq{b}$.
Minimality and
transitivity give a representation of $s$ in that tail; if it is a
singleton, move beyond that term and represent $s$ again.
Finally, (2)
places $I$ inside $\langle\myseq{b}\rangle$, while the two absorption
inclusions are inherited from $\langle\myseq{c}\rangle$.
This proves
(3).
\end{proof}

\begin{lemma}\label{lem:cyclic-embedding}
Let $\myseq{b}$ be an injective sequence and let $I:=I(\myseq{b})$.
Assume that every increasing ordered sum of at least
two terms of $\myseq{b}$ is determined by its length and that
$\FS_{\ge2}(b_n,b_{n+1},\dotsc)=I$ for all $n$.
Then there are a cyclic semigroup $\langle a\rangle$ and an embedding
$\theta\colon I\to\langle a\rangle$ such that
\begin{equation}
\label{theta-length}
\theta(b_1+\dotsb+b_n)=n\cdot a
\end{equation}
for all $n\ge2$.
In particular, every element of $\theta[I]$ has the form $n\cdot a$
with $n\ge2$.
\end{lemma}

\begin{proof}
Define $\phi\colon\bbN_{\ge2}\longrightarrow I$ by
$\phi(n):=b_1+\dotsb+b_n$.
Length determinacy gives
$\phi(m)=b_{n+1}+\dotsb+b_{n+m}$, and hence
$\phi(n)+\phi(m)=\phi(n+m)$.
If $s\in I$, represent it in a tail
containing no term equal to $s$.
The representation has length at least
two, so $s=\phi(m)$ for some $m$; thus $\phi$ is surjective.
Moreover, if $\phi(n)=\phi(m)=s$, then length determinacy gives
$b_2+\dotsb+b_{n+1}=s=b_2+\dotsb+b_{m+1}$, and therefore
$\phi(n+1)=b_1+s=\phi(m+1)$.

Define an equivalence relation on $\bbN$ by
\[
n\sim m
\quad\Longleftrightarrow\quad
n=m=1\ \text{ or }\ n,m\ge2\text{ and }\phi(n)=\phi(m).
\]
The homomorphism property handles addition of equivalent numbers at
least $2$, and the displayed shift property handles addition of $1$;
thus $\sim$ is a congruence.
Let $a:=[1]$ in $\bbN/\!\sim$.
For
$s\in I$, choose $n\ge2$ with $\phi(n)=s$ and let
$\theta(s):=[n]$.
This is well-defined and injective by the definition
of $\sim$, and
$\theta(\phi(n)+\phi(m))=[n+m]=[n]+[m]$,
so it is a homomorphism.
Equation~\eqref{theta-length} follows.
\end{proof}

\begin{corollary}
\label{cor:fs2}
Let $\myseq{a}$ be a minimal injective sequence with no proper
sumsequence in a semigroup $S$, and $I:=I(\myseq{a})$.
Then there
exist a cyclic semigroup $\langle a\rangle$, an injective sumsequence
$\myseq{b}$ of $\myseq{a}$, and an embedding
$\theta\colon I\to\langle a\rangle$ such that:
\begin{enumerate}
\item $I$ is a two-sided ideal of $\langle\myseq{b}\rangle$;
\item for all $n<m$, $b_n+b_m\in I$ and
$\theta(b_n+b_m)=2\cdot a$;
\item For every $s\in I$ and every $k\ge2$ such that
$\theta(s)=k\cdot a$, we have
\[
\theta(b_n+s)=\theta(s+b_n)=(k+1)\cdot a
\]
for all $n$.
\item Every element of $\theta[I]$ has the form $k\cdot a$ with
$k\ge2$.
\end{enumerate}
\end{corollary}

\begin{proof}
Use Lemma~\ref{lem:fs2}, and call the resulting sequence $\myseq{c}$.
It is minimal and has tail intersection $I$, by minimality and
transitivity.
Apply Lemma~\ref{lem:stable-ideal-tail} and call the
  resulting tail $\myseq{d}$.
  By minimality and transitivity, the sequence $\myseq{d}$ is minimal and
  $I(\myseq{d})=I$.
  Apply Lemma~\ref{lem:length-determined} to
  $\myseq{d}$ and call the resulting subsequence $\myseq{b}$.
Finally,
apply Lemma~\ref{lem:cyclic-embedding} and fix the resulting cyclic
semigroup $\langle a\rangle$ and embedding
$\theta\colon I\to\langle a\rangle$.

\begin{claim}\label{clm:action-tail}
Some tail $\myseq{t}$ of $\myseq{b}$ satisfies
$t_n+t_m\in I$ and $\theta(t_n+t_m)=2\cdot a$ for all $n<m$.
Moreover, for every $s\in I$, every $k\ge2$ with
$\theta(s)=k\cdot a$, we have
\[
\theta(t_n+s)=\theta(s+t_n)=(k+1)\cdot a
\]
for all $n$.
The set $I$ remains a two-sided ideal of $\langle\myseq{t}\rangle$.
\end{claim}

\begin{proof}[Proof of the claim]
The first two assertions follow from length determinacy and
Equation~\eqref{theta-length}.
Fix $s\in I$ and $k\ge2$ with $\theta(s)=k\cdot a$.
Given $n\in\bbN$, choose indices $n<i_1<\dotsb<i_k$.
The sum $r:=b_{i_1}+\dotsb+b_{i_k}$ belongs to $I$, and length
determinacy and Equation~\eqref{theta-length} give
$\theta(r)=k\cdot a=\theta(s)$.
Since $\theta$ is injective, $r=s$, and therefore
$\theta(b_n+s)=(k+1)\cdot a$.

Since $I$ is finite, choose $m$ such that every element of $I$ has a
representation by at least two terms among $b_1,\dotsc,b_{m-1}$.
Fix such a representation of $s$, and let $j$ be its length.
Then $\theta(s)=j\cdot a$.
Thus $j\cdot a=k\cdot a$, and the congruence
used in Lemma~\ref{lem:cyclic-embedding} gives
$(j+1)\cdot a=(k+1)\cdot a$.
For $n\ge m$, length determinacy now
gives $\theta(s+b_n)=(k+1)\cdot a$.
Let $\myseq{t}$ be the tail
beginning at $m$.
The ideal inclusions are inherited, and since
$\FS_{\ge2}(t_n,t_{n+1},\dotsc)=I$ for all $n$, we have
$I\subseteq\langle\myseq{t}\rangle$.
\end{proof}

Fix $\myseq{t}$ as in Claim~\ref{clm:action-tail} and relabel it as
$\myseq{b}$.
The four conclusions now follow from the claim and
Lemma~\ref{lem:cyclic-embedding}.
Every passage was to a subsequence
or a tail, so transitivity gives an injective sumsequence of the
original sequence.
\end{proof}

The following is our main result.

\begin{theorem}
\label{mr1}
The following assertions are equivalent for a semigroup $S$:
\begin{enumerate}
\item The proper IP sets are partition regular.
\item Every injective sequence has a proper sumsequence.
\item The semigroup $S$ has no subsemigroup of any of the following
  types:
  \begin{enumerate}
  \item an infinite semigroup $T$ with $T+T$ finite;
  \item the fan semilattice $(\bbN,\wedge)$, where
    $m\wedge n=1$ for distinct $m,n$;
  \item the semigroup $(\bbN\times\bbN)\cup\{0\}$ in which $0$ is
    absorbing and
    \[
    (m_1,n_1)+(m_2,n_2):=
    \begin{cases}
    (m_1,n_1+n_2),&m_1=m_2,\\
    0,&m_1\ne m_2.
    \end{cases}
    \]
  \end{enumerate}
\item For every finite coloring of $S$ and every injective sequence
  $\myseq{a}$, some injective sumsequence $\myseq{b}$ of $\myseq{a}$
  has $\FS(\myseq{b})$ monochromatic.
\item For every finite coloring of $S$ and every injective sequence
  $\myseq{a}$, some proper sumsequence $\myseq{b}$ of $\myseq{a}$ has
  $\FS(\myseq{b})$ monochromatic.
\end{enumerate}
\end{theorem}

\begin{proof}
\begin{claim}\label{clm:forbidden-obstructions}
Each of the three semigroups listed in assertion~(3) has a proper IP
set with a finite partition none of whose cells is a proper IP set.
\end{claim}

\begin{proof}[Proof of the claim]
Let $T$ be infinite with $T+T$ finite.
It is a proper IP set, since it
contains the finite-sums set of any injective sequence in $T$.
In
\[
T=(T\setminus(T+T))\cup(T+T),
\]
the second cell is finite, and the first contains no proper IP set
because the sum of the first two terms of any witnessing sequence lies
in $T+T$.

In the fan semilattice, the sequence $2,3,4,\dotsc$ has finite-sums set
$\bbN$.
Neither cell of the partition
$(\bbN\setminus\{1\})\cup\{1\}$ contains a proper IP set.

In type~(c), the set
$\{0\}\cup\{(n,1):n\in\bbN\}=\FS((1,1),(2,1),(3,1),\dotsc)$
is proper.
Partition it into its nonzero elements and $\{0\}$.
Two
distinct nonzero elements sum to $0$, so neither cell contains a
proper IP set.
\end{proof}

$(2\Rightarrow5)$
Assume (2), let $\myseq{a}$ be injective, and choose a proper
sumsequence $\myseq{c}$.
Theorem~\ref{thm:FGH} gives a sumsequence
$\myseq{b}$ of $\myseq{c}$ with monochromatic finite-sums set.
A sumsequence of a proper sequence is proper.
Thus
(5) holds.

$(5\Rightarrow4)$ Every proper sequence is injective.

$(4\Rightarrow1)$
Assume (4), let $A$ be a proper IP set, and choose an injective sequence
$\myseq{a}$ with $\FS(\myseq{a})\subseteq A$.
Extend a finite coloring
of $A$ arbitrarily to $S$.
Assertion (4) gives an injective sumsequence
whose finite-sums set is a monochromatic proper IP subset of $A$.

Claim~\ref{clm:forbidden-obstructions} proves $(1\Rightarrow3)$, since
a partition witnessing failure inside a subsemigroup also witnesses
failure in $S$.

$(3\Rightarrow2)$ Suppose, towards a contradiction, that the semigroup
$S$ has an injective sequence with no proper sumsequence.
Apply
Proposition~\ref{prp:minsubsemi}, and then Corollary~\ref{cor:fs2}, to
obtain an injective sequence $\myseq{x}$ with finite tail intersection
$I$, a cyclic semigroup $\langle a\rangle$, and an embedding of $I$
satisfying the four conclusions of that corollary.
Identify $I$ with
its image.
For the remainder of the proof, fix these objects.
Thus $I$
is a finite two-sided ideal of $\langle\myseq{x}\rangle$,
$2\cdot a\in I$, every element of $I$ has the form $k\cdot a$ with
$k\ge2$, and
\begin{equation}
\label{forwardaction}
\begin{aligned}
x_n+x_m&=2\cdot a &&(n<m),\\
x_n+k\cdot a&=k\cdot a+x_n=(k+1)\cdot a
&& (n,k\in\bbN,\ k\cdot a\in I).
\end{aligned}
\end{equation}

\begin{claim}\label{clm:reverse-reduction}
Either $S$ contains an infinite subsemigroup $T$ with $T+T$ finite, or
$\myseq{x}$ has a subsequence $\myseq{c}$ for which there is a fixed
$b\in S$ satisfying
\begin{equation}
\label{reverseconstant}
c_m+c_n=b
\end{equation}
for all $n<m$.
\end{claim}

\begin{proof}[Proof of the claim]
For a natural number $n$, let
$U_n:=\{x_m+x_n:m>n\}$.
If some $U_n$ is infinite, let $T:=U_n\cup I$.
Products involving $I$ lie in $I$, and, for $m_1,m_2>n$,
$(x_{m_1}+x_n)+(x_{m_2}+x_n)
=x_{m_1}+(x_n+x_{m_2})+x_n\in I$.
Thus $T$ is infinite and $T+T\subseteq I$.

Suppose that every $U_n$ is finite.
Set $n_1:=1$ and choose an
infinite fiber $M_1\subseteq\{m:m>n_1\}$ of the map
$m\mapsto x_m+x_{n_1}$.
Having chosen $n_k$ and $M_k$, take $n_{k+1}\in M_k$ and choose an
infinite fiber $M_{k+1}\subseteq M_k\cap\{m:m>n_{k+1}\}$ of
$m\mapsto x_m+x_{n_{k+1}}$.
Since $M_k\subseteq\{m:m>n_k\}$, the indices $n_1<n_2<\dotsb$ are
strictly increasing.
The sets $M_k$ decrease, so the sequence
$y_k:=x_{n_k}$ has the property that $y_m+y_n$ is independent of
$m>n$.
Denote this value by $b_n$.

If $B:=\{b_n:n\in\bbN\}$ is infinite, let $T:=B\cup I$.
Given $n,i$, choose $j$ with $\{n+1,i\}<j$.
Then $b_n+b_i=(y_{n+1}+y_n)+(y_j+y_i)
=y_{n+1}+(y_n+y_j)+y_i\in I$, so again $T+T\subseteq I$.
If $B$ is finite, the infinite pigeonhole
principle gives an infinite set of indices on which $b_n$ is constant.
The corresponding subsequence satisfies
Equation~\eqref{reverseconstant}.
\end{proof}

\begin{claim}\label{clm:mixed-words}
Let $\myseq{c}$ be a subsequence of $\myseq{x}$ satisfying
Equation~\eqref{reverseconstant}, and let $d_n:=c_{n+1}$.
Then
$d_m+d_n=b$
for all $n<m$,
and
\begin{equation}
\label{baction}
b+d_n=d_n+b=3\cdot a
\end{equation}
for all $n$.
Moreover, if $k\ge3$ and $n_1,\dotsc,n_k$ are not all equal, then
\begin{equation}
\label{mixedword}
d_{n_1}+\dotsb+d_{n_k}=k\cdot a.
\end{equation}
\end{claim}

\begin{proof}[Proof of the claim]
Associativity and Equation~\eqref{forwardaction} give
\[
b+d_n=(c_{n+1}+c_n)+c_{n+1}=3\cdot a
\]
and
\[
d_n+b=c_{n+1}+(c_{n+2}+c_{n+1})=3\cdot a.
\]
For the second assertion, choose adjacent unequal indices.
If
$n_i<n_{i+1}$, replace the corresponding pair by $2\cdot a$, then use
Equation~\eqref{forwardaction} to append the terms on its right and
prepend those on its left.
Each step raises the coefficient by one.

If $n_i>n_{i+1}$, the adjacent pair equals $b$.
If the pair has a term on its right, combine it with that term; otherwise,
combine it with the term on its left.
Equation~\eqref{baction} turns
the resulting three-term word into $3\cdot a$.
Process the remaining
terms with
Equation~\eqref{forwardaction}.
In either case the result is
$k\cdot a$.
\end{proof}

For $s\in S$, write
$C(s):=\{k\cdot s:k\in\bbN\}$.

\begin{claim}\label{clm:small-hitting-times}
Let $\myseq{d}$ be an injective sequence satisfying
  $d_m+d_n=b$ for $n<m$ and
  Equations~\eqref{baction} and~\eqref{mixedword}.
  Suppose that, for every $n$, there is a positive integer $k$ such that
  $k\cdot d_n\in I$.
  For each $n$, let $m_n$ be the least such integer, and suppose that
  $m_n\le3$ for every $n$.
  Then $S$ contains an infinite subsemigroup $T$ with $T+T$ finite.
\end{claim}

\begin{proof}[Proof of the claim]
If one value occurs infinitely often among the doubles
$2\cdot d_n$, pass to a subsequence on which all doubles are equal.
Otherwise choose a subsequence with pairwise distinct doubles.

In the first case, write $2\cdot d_n=h$ for every $n$ and let
$T:=\{d_n:n\in\bbN\}\cup\{h,b\}\cup I$.
For all $n,j$, the sum $d_n+d_j$ is $h$, $2\cdot a$, or $b$, according
as $n=j$, $n<j$, or $n>j$.
Equation~\eqref{baction} puts every product of
$b$ with a $d_n$ in $I$.
Using $h=2\cdot d_i$ with $i\ne n$ and
Equation~\eqref{mixedword} gives $h+d_n,d_n+h\in I$.
Since
$m_1\le3$ and $m_1\cdot d_1\in I$, the ideal property gives
$h+h=4\cdot d_1\in I$.
Representing $b$ by $d_2+d_1$ and using
Equation~\eqref{mixedword} puts $b+b,h+b,b+h$ in $I$.
Hence
$T+T\subseteq\{h,b\}\cup I$.

In the second case, let
$T:=\{2\cdot d_n:n\in\bbN\}\cup I$.
This set is infinite.
The ideal property gives $4\cdot d_n\in I$, and
Equation~\eqref{mixedword} gives
$2\cdot d_n+2\cdot d_j\in I$ for $n\ne j$.
Thus $T+T\subseteq I$.
\end{proof}

\begin{claim}\label{clm:hitting-ideal}
Let $\myseq{d}$ be an injective sequence satisfying
  $d_m+d_n=b$ for $n<m$ and
  Equations~\eqref{baction} and~\eqref{mixedword}.
  If
  $C(d_n)\cap I\ne\emptyset$ for every $n$, then $S$ contains an
  infinite subsemigroup $T$ with $T+T$ finite.
\end{claim}

\begin{proof}[Proof of the claim]
Let $m_n$ be the least positive integer with $m_n\cdot d_n\in I$.
Pass to a subsequence on which either every $m_n\ge4$ or every
$m_n\le3$.
In the second case, Claim~\ref{clm:small-hitting-times}
applies.

Suppose that every $m_n\ge4$.
Let
$k_n:=\lceil m_n/2\rceil$ and $s_n:=k_n\cdot d_n$.
These elements are
pairwise distinct.
Indeed, if $s_n=s_j$ with $n\ne j$, then
$(k_n+1)\cdot d_n=d_n+k_j\cdot d_j\in I$
by Equation~\eqref{mixedword}, whereas $k_n+1<m_n$.
Also
$2k_n\cdot d_n\in I$ by the ideal property, and, for $n\ne j$,
Equation~\eqref{mixedword} gives $s_n+s_j\in I$.
Therefore
$T:=\{s_n:n\in\bbN\}\cup I$
is infinite and satisfies $T+T\subseteq I$.
\end{proof}

\begin{claim}\label{clm:finite-cyclic-fan}
Let $\myseq{d}$ satisfy Equation~\eqref{mixedword}.
If every $C(d_n)$
is finite and disjoint from $I$, then $S$ contains a copy of the fan
semilattice.
\end{claim}

\begin{proof}[Proof of the claim]
Choose an idempotent $e_n=m_n\cdot d_n$ in each finite cyclic
semigroup $C(d_n)$.
Choose an idempotent $e\in I$ and write
$e=k\cdot a$ with $k\ge2$.
For $n\ne j$, idempotence and Equation~\eqref{mixedword} give
$e_n+e_j=k\cdot e_n+k\cdot e_j=(m_n+m_j)\cdot e=e$.

We also have $e_n+e=e+e_n=e$.
To see this, choose $k$ distinct terms of
$\myseq{d}$, none equal to $d_n$.
Their increasing sum is
$k\cdot a=e$.
Replace $e_n$ by $k\cdot e_n$ and expand either product.
The resulting mixed word has length $k(m_n+1)$, so
Equation~\eqref{mixedword} gives $(m_n+1)\cdot e=e$.

The $e_n$ are pairwise distinct: equality of two would give
$e_n=e_n+e_j=e\in I$.
They are also distinct from $e$.
Therefore the
map sending the center of the fan to $e$ and its other elements
injectively to the $e_n$ is an embedding.
\end{proof}

\begin{claim}\label{clm:free-cyclic-type-c}
Let $\myseq{d}$ satisfy Equation~\eqref{mixedword}, and suppose that
every $C(d_n)$ is isomorphic to $(\bbN,+)$ and disjoint from $I$.
Then
$S$ contains a copy of the semigroup in assertion~(3c).
\end{claim}

\begin{proof}[Proof of the claim]
The cyclic subsemigroups $C(d_n)$ are pairwise disjoint.
If $n<j$ and
$m\cdot d_n=i\cdot d_j$, then
$(m+1)\cdot d_n=d_n+i\cdot d_j\in I$,
using $d_n+d_j=2\cdot a\in I$ and the ideal property.
This contradicts
$C(d_n)\cap I=\emptyset$.

Choose an idempotent $e=k\cdot a\in I$ with $k\ge2$.
If $n\ne j$ and $m,i\ge1$, Equation~\eqref{mixedword} gives
$mk\cdot d_n+ik\cdot d_j=(m+i)\cdot e=e$.
As in Claim~\ref{clm:finite-cyclic-fan}, representing $e$ by $k$
distinct terms different from $d_n$ gives
$mk\cdot d_n+e=e+mk\cdot d_n=e$.
Define
$\psi(0):=e$ and $\psi(n,m):=mk\cdot d_n$.
Within each row this respects addition; different rows sum to $e$, and
$e$ is absorbing.
The map is injective because each $C(d_n)$ is free,
different rows are disjoint, and no nonzero image lies in $I$.
\end{proof}

\begin{claim}\label{clm:forbidden-extraction}
Under the standing hypotheses above, the semigroup $S$ contains one of
the three subsemigroups listed in assertion~(3).
\end{claim}

\begin{proof}[Proof of the claim]
Apply Claim~\ref{clm:reverse-reduction}.
If it produces an infinite
semigroup $T$ with $T+T$ finite, the first alternative holds.
Otherwise
obtain a subsequence $\myseq{c}$ and $b$ satisfying
Equation~\eqref{reverseconstant}.
Claim~\ref{clm:mixed-words} supplies
a tail $\myseq{d}$ on which Equations~\eqref{baction} and
\eqref{mixedword} hold.

A cyclic semigroup $C(s)$ is finite if two positive multiples of $s$
agree; otherwise the map $k\mapsto k\cdot s$ is an isomorphism from
$(\bbN,+)$ onto $C(s)$.
Partition the terms of $\myseq{d}$ according
as $C(d_n)$ meets $I$, is finite and disjoint from $I$, or is
isomorphic to $(\bbN,+)$ and disjoint from $I$.
Pass to a subsequence
lying in one class.
The three cases are handled, respectively, by
Claim~\ref{clm:hitting-ideal}, Claim~\ref{clm:finite-cyclic-fan}, and
Claim~\ref{clm:free-cyclic-type-c}.
\end{proof}

Claim~\ref{clm:forbidden-extraction} gives one of the forbidden
subsemigroups, contradicting assertion~(3).
This completes the proof.
\end{proof}

\section{A solution of a problem of Andrews and Goldbring}
\label{sec:moving}

An infinite semigroup $S$ is \emph{moving} if, for each infinite set
$A\subseteq S$ and each finite set $F\subseteq S$, there exist a natural
number $k$ and elements $a_1,\dotsc,a_k\in A$ such that
$\{a_1+s,\dotsc,a_k+s\}\nsubseteq F$
for all but finitely many $s\in S$.
Golan and Tsaban~\cite{TG} proved that in every moving semigroup, the proper IP sets are partition regular.
Andrews and Goldbring~\cite{AG} asked whether partition regularity of proper IP sets characterizes moving semigroups.
Our main theorem implies the following answer.

\begin{corollary}
There is a non-moving semigroup $S$ in which the proper IP sets are partition regular.
\end{corollary}
\begin{proof}
Steinberg~\cite{Ste} proved that the following semigroup has finite-to-one right addition, but is not moving:
\[S := \langle \, t, x_0, x_1, \dotsc : x_0 + t = x_0, x_i + t = x_{i-1}, i>0 \, \rangle. \]
Finite-to-one right addition passes to subsemigroups.
None of the three
forbidden semigroups has this property.
In type~(a), for a fixed $t\in T$
the map $x\mapsto x+t$ has finite image contained in $T+T$ and therefore
has an infinite fiber.
In the fan semilattice, right addition by $2$ maps
every element other than $2$ to $1$.
In type~(c), right addition by
$(1,1)$ maps every $(m,n)$ with $m\ne1$ to $0$.
Hence the semigroup
displayed above contains no forbidden subsemigroup.
By
Theorem~\ref{mr1}, its proper IP sets are partition regular.
\end{proof}

\section{Hindman's theorem with infinitely many proper IP sets}
\label{sec:manyIP}

We strengthen the existence of a monochromatic proper IP set to infinitely
many pairwise disjoint such sets.

\begin{definition}
A finite sequence $s_1,\dotsc,s_k$ in a semigroup is
\emph{separated} if
\[
\FS(s_1,\dotsc,s_i)\cap (\FS(s_1,\dotsc,s_i)+s_{i+1})=\emptyset
\]
for all $i=1,\dotsc,k-1$.
An infinite sequence is separated if each of its finite initial
segments is separated.
\end{definition}

\begin{lemma}\label{lem:absorbing-stage}
Let $\myseq{a}$ be a sequence with no separated sumsequence.
For every natural number $n$, there are nonempty finite index sets
$n\le F<G_1<G_2<\dotsb$
such that, for
$c:=a_F$ and $t_i:=a_{G_i}$ ($i=1,2,\dotsc$),
we have
$c+s=c$
for all $s\in\FS(\myseq{t})$.
\end{lemma}
\begin{proof}
Fix $n$ and work in the tail $a_n,a_{n+1},\dotsc$.
There are nonempty finite index sets
$F_1<\dotsb<F_k$ such that, with $s_i:=a_{F_i}$, the finite sequence
$s_1,\dotsc,s_k$ is separated and is maximal under extension by a
term $a_F$ with $F_k<F$.
Otherwise, recursively adjoining such terms would produce a separated sumsequence.
Choose $m>F_k$.

For each
$z\in\FS(a_m,a_{m+1},\dotsc)$, choose a nonempty finite index set
$H\ge m$ such that $z=a_H$.
By maximality, choose
$x(z),y(z)\in\FS(s_1,\dotsc,s_k)$ such that
$x(z)+z=y(z)$.
Use the pair $(x(z),y(z))$ as the color of $z$.
Extend this coloring arbitrarily to the ambient semigroup and apply
Theorem~\ref{thm:FGH} to the sequence
$a_m,a_{m+1},\dotsc$.
We obtain a sumsequence
$\myseq{u}$ and fixed $x,y$ such that
$x+u_H=y$ for every nonempty finite $H$.
In particular, $x+u_1=y$.
If $z\in\FS(u_2,u_3,\dotsc)$, then
\[
y+z=(x+u_1)+z=x+(u_1+z)=y.
\]
Choose a nonempty set $J\subseteq\{1,\dotsc,k\}$ such that
$y=s_J$, and let
$F:=\bigcup_{i\in J}F_i$ and $c:=a_F=y$.
Write
$u_i=a_{H_i}$
for all $i$,
where $m\le H_1<H_2<\dotsb$
are nonempty finite index sets.
For each $i$, let
$G_i:=H_{i+1}$ and $t_i:=u_{i+1}=a_{G_i}$.
Since the sets $F_1,\dotsc,F_k$ lie in the tail beginning at $n$ and
$m>F_k$, we have
$n\le F<G_1<G_2<\dotsb$.
The preceding calculation gives
$c+s=c$
for all $s\in\FS(\myseq{t})$.
\end{proof}

\begin{lemma}
\label{lem:sumseq}
Every sequence in a semigroup has a sumsequence $\myseq{b}$ of one of
the following types:
\begin{enumerate}
\item $b_n+b_m=b_n$ for all $n<m$;
\item
\[
\FS(b_1,\dotsc,b_n)\cap
\bigl(\FS(b_1,\dotsc,b_n)+b_{n+1}\bigr)=\emptyset
\]
for all $n$.
\end{enumerate}
\end{lemma}

\begin{proof}
Let $\myseq{a}$ be the given sequence.
If $\myseq{a}$ has a separated sumsequence, there is nothing to prove.
Otherwise construct nested sumsequences
$\myseq{b^{(1)}},\myseq{b^{(2)}},\dotsc$ and finite sums
$c_1,c_2,\dotsc$ with successive representations in the original
sequence.
Set $\myseq{b^{(1)}}:=\myseq{a}$.
At stage $n$, apply Lemma~\ref{lem:absorbing-stage} to
$\myseq{b^{(n)}}$, with initial index $1$.
Let $c_n$ and $\myseq{b^{(n+1)}}$ be the finite sum and the later
sumsequence supplied by that lemma.
Then
$c_n+s=c_n$ for all $s\in\FS(\myseq{b^{(n+1)}})$.
The index sets representing $c_1,c_2,\dotsc$ are successive,
so $\myseq{c}$ is a sumsequence of the original sequence.
If $n<m$, nesting gives
$c_m\in\FS(\myseq{b^{(n+1)}})$, and hence $c_n+c_m=c_n$.
\end{proof}

\begin{proposition}
\label{prp:D}
Every proper sequence in a semigroup has a sumsequence $\myseq{b}$ such that
$b_F\neq b_G$ for all disjoint nonempty finite index sets $F$ and $G$.
\end{proposition}
\begin{proof}
Let $\myseq{a}$ be proper, and choose a sumsequence $\myseq{b}$ of one
of the two types in Lemma~\ref{lem:sumseq}.
Every sumsequence of a
proper sequence is proper and hence injective.

Suppose first that $b_n+b_m=b_n$ whenever $n<m$.
For every nonempty
finite $F$, choose $i\in F$ with $i\le F$.
Induction on $|F|$ gives
$b_F=b_i$.
If $F$ and $G$ are disjoint, the corresponding indices are
distinct, and injectivity gives $b_F\ne b_G$.

Now suppose that $\myseq{b}$ is separated.
Interchanging $F$ and $G$ if
necessary, choose $k\in G$ with $F<k$ and $G\le k$, and let
$H:=G\setminus\{k\}$.
If $H=\emptyset$, then $F<G$, so properness
gives $b_F\ne b_G$.
Otherwise, $H<k$, and therefore
$b_F\in\FS(b_1,\dotsc,b_{k-1})$ and $b_G=b_H+b_k\in\FS(b_1,\dotsc,b_{k-1})+b_k$.
These two sets are disjoint by the separation condition at $k-1$, and hence
$b_F\ne b_G$.
\end{proof}

\begin{corollary}
\label{cor:inf}
Let $\myseq{a}$ be a proper sequence in a semigroup.
The IP set $A=\FS(\myseq{a})$ can be partitioned into infinitely many proper IP sets.
\end{corollary}
\begin{proof}
Choose $\myseq{b}$ as in Proposition~\ref{prp:D}.
Partition
$\bbN$ into infinite sets $I_1,I_2,\dotsc$.
For each $n$, let
$A_n:=\FS\bigl((b_i)_{i\in I_n}\bigr)$,
where the subsequence is taken in increasing order.
Each $A_n$ is a
proper IP set, because every subsequence of the proper sequence
$\myseq{b}$ is proper.
The sets are pairwise disjoint: if
$x\in A_n\cap A_m$ for $n\ne m$, then $x=b_F=b_G$ for nonempty finite
sets $F\subseteq I_n$ and $G\subseteq I_m$; these sets are disjoint,
contrary to Proposition~\ref{prp:D}.

The union of the $A_n$ need not be all of $A$.
Replace $A_1$ by
$A_1\cup\left(A\setminus\bigcup_{n=1}^{\infty}A_n\right)$.
This enlargement remains a proper IP set, and the resulting sets form
a partition of $A$ into infinitely many proper IP sets.
\end{proof}

\begin{theorem}\label{thm:manyIP}
Suppose that every finite coloring of a semigroup $S$ has a
monochromatic proper IP set.
Then every finite coloring of $S$ has
infinitely many pairwise disjoint monochromatic proper IP sets.
\end{theorem}

\begin{proof}
Fix a finite coloring $c$ of $S$.
If $S$ contains a proper sequence,
apply Theorem~\ref{thm:FGH} to it.
The resulting sumsequence is proper, so its finite-sums set is a
monochromatic proper IP set.
By Corollary~\ref{cor:inf}, this set can be
partitioned into infinitely many pairwise disjoint proper IP sets,
which remain monochromatic.

Suppose that $S$ has no proper sequence.

\begin{claim}\label{clm:extend-disjoint-family}
Suppose that $A_1,\dotsc,A_m$ are pairwise disjoint,
$c$-monochromatic proper IP sets with representations by injective
sequences
$A_i=\FS(a^{(i)}_1,a^{(i)}_2,\dotsc)$
for which
$B_i:=\FS_{\ge2}(a^{(i)}_1,a^{(i)}_2,\dotsc)$
is finite.
Then there is another such set $A_{m+1}$, disjoint from all
preceding ones.
\end{claim}

\begin{proof}[Proof of the claim]
Let $A:=\bigcup_{i=1}^m A_i$ and $B:=\bigcup_{i=1}^m B_i$, with both
sets empty when $m=0$.
Choose two new colors $\alpha,\beta$ and define
\[
c'(s):=
\begin{cases}
c(s),&s\notin A,\\
\alpha,&s\in A\setminus B,\\
\beta,&s\in B.
\end{cases}
\]
Choose an injective sequence $\myseq{b}$ whose finite-sums set lies in
one $c'$-color class.
That color is not $\beta$, since $B$ is finite.
It is not $\alpha$ either.
If $m=0$, the $\alpha$-color class is empty.
If $m>0$ and $\FS(\myseq{b})$ has color $\alpha$, infinitely many
singleton terms $b_n$ lie in one $A_i\setminus B_i$.
Each is therefore equal to a term of the injective sequence
$\myseq{a^{(i)}}$.
Pass to a subsequence on
which the corresponding indices increase.
The sum of any two terms
belongs to $B_i$ and has color $\beta$, contradicting monochromaticity.

Thus $\FS(\myseq{b})\subseteq S\setminus A$, where the colorings
$c'$ and $c$ agree.
Since $S$
has no proper sequence, $\myseq{b}$ has no proper sumsequence.
Lemma~\ref{lem:fs2} gives an injective sumsequence $\myseq{s}$ with
$\FS_{\ge2}(\myseq{s})$ finite.
Set $A_{m+1}:=\FS(\myseq{s})$ and
$B_{m+1}:=\FS_{\ge2}(\myseq{s})$.
These sets have all the required
properties.
\end{proof}

Start with the empty family and apply
Claim~\ref{clm:extend-disjoint-family} recursively.
This produces the
required infinite pairwise disjoint family.
\end{proof}

The second case of the last proof cannot be omitted, as the following
example shows.
For each $n$, let $(T_n,\wedge)$ be the semilattice of
words over $\bbN$ of length at most $n$, including the empty word
$\varepsilon$, where $u\wedge v$ is the longest common prefix of $u$
and $v$.

\begin{lemma}\label{lem:word-semilattices}
For every natural number $n$, the semilattice $T_n$ has no proper sequence, and
every coloring of $T_n$ with at most $n$ colors has a monochromatic
proper IP set.
\end{lemma}

\begin{proof}
We argue by induction on $n$.
For $n=1$, discard a possible occurrence
of the empty word from an injective sequence.
Among four remaining
one-letter words, the meet of the first pair and that of the second pair
are both $\varepsilon$, so the sequence is not proper.
In a one-coloring,
the sequence of distinct one-letter words has monochromatic finite-sums
set consisting of those words and $\varepsilon$.

Assume the assertions for $T_n$.
For a natural number $j$, let
$B_j:=\{jw:w\in T_n\}\subseteq T_{n+1}$.
Each $B_j$ is a copy of $T_n$, and elements from different $B_j$'s meet
at $\varepsilon$.
An injective sequence in $T_{n+1}$ either has an
infinite subsequence in one $B_j$, contradicting the induction
hypothesis if it is proper, or has four terms in four different
$B_j$'s.
In the latter case the meets of the first and second pairs are
both $\varepsilon$.
Thus $T_{n+1}$ has no proper sequence.

Color $T_{n+1}$ with at most $n+1$ colors, and let $\alpha$ be the color
of $\varepsilon$.
If some $B_j$ omits $\alpha$, its restriction uses at
most $n$ colors and the induction hypothesis applies.
Otherwise choose
an $\alpha$-colored $x_j\in B_j$ for every $j$.
Every sum of at least
two of the distinct elements $x_j$ is $\varepsilon$, so
$\FS(x_1,x_2,\dotsc)=\{\varepsilon,x_1,x_2,\dotsc\}$
is a monochromatic proper IP set.
\end{proof}

\begin{proposition}
Consider the semigroup
\[
S:=\{0\}\cup\{(n,w):n\in\bbN,\ w\in T_n\},
\]
where $0$ is absorbing, $(n,u)+(n,v)=(n,u\wedge v)$, and
$(n,u)+(m,v)=0$ for $n\ne m$.
The semigroup $S$ has no proper sequence,
but every finite coloring of $S$ has a monochromatic proper IP set.
\end{proposition}

\begin{proof}
The operation is associative.
Inside one copy of $T_n$ this follows
from associativity of longest common prefix; a product involving $0$ or
two different copies is $0$ under either parenthesization.

If $S$ had a proper sequence, then after discarding a possible term
$0$, either infinitely many terms would lie in one copy of some $T_n$,
or four terms could be chosen from four different copies.
The first
alternative contradicts Lemma~\ref{lem:word-semilattices}; in the
second, the sums of the first and second pairs are both $0$.

Finally, restrict a coloring of $S$ using $n$ colors to the copy
$\{(n,w):w\in T_n\}$.
Lemma~\ref{lem:word-semilattices} supplies a
monochromatic proper IP set in that copy.
\end{proof}

\subsection*{Acknowledgments}
This paper is based on a part of Yonatan Gadot's M.Sc.\ thesis, supervised by Boaz Tsaban.
We thank the referees for their thorough and useful reports.
The final paper revision was aided by ChatGPT.
 This paper was verified in Lean using ChatGPT Codex~\cite{lean}.

\end{document}